\documentclass[a4paper]{article}
\usepackage{amsthm,amssymb,amsmath,enumerate,graphicx,epsf}

\newcommand{\COLORON}{0}
\newcommand{\NOTESON}{0}
\newcommand{\Debug}{0}
\usepackage[usenames,dvips]{color} 
\usepackage{amsthm,amssymb,amsmath,bbm,enumerate,graphicx,epsf,stmaryrd}
\usepackage[dvips, bookmarks, colorlinks=false, breaklinks=true]{hyperref} 

\newcommand{\comment}[1]{}
\newcommand{\COMMENT}[1]{}

\definecolor{darkgray}{rgb}{0.3,0.3,0.3}
\newcommand{\defi}[1]{{\color{darkgray}\emph{#1}}}

\newcommand{\acknowledgement}{\section*{Acknowledgement}}

\comment{
	\begin{lemma}\label{}	
\end{lemma}
\begin{proof}

\end{proof}

\begin{theorem}\label{}
\end{theorem} 
\begin{proof} 	

\end{proof}

}

\newtheorem{proposition}{Proposition}[section]
\newtheorem{definition}[proposition]{Definition}
\newtheorem{theorem}[proposition]{Theorem}

\newtheorem{lemma}[proposition]{Lemma}

\newtheorem{conjecture}{{Conjecture}}[section]

\newtheorem{problem}[conjecture]{{Problem}}

\newtheorem{examp}[proposition]{Example}

\newcommand{\FIG}{0}

\ifnum \NOTESON = 1 \newcommand{\note}[1]{ 

\hspace*{-30pt}
	{\color{blue}  NOTE: \color{Turquoise}{\small  \tt \begin{minipage}[c]{1.1\textwidth}  #1 \end{minipage} \ignorespacesafterend }} 
	
	}
\else \newcommand{\note}[1]{} \fi

\newcommand{\afsubm}[1]{ \ifnum \Debug = 1 {\mymargin{#1}}
\fi} 

\ifnum \Debug = 1 
\else  \fi

\ifnum \FIG = 1 
\else  \fi

\ifnum \FIG = 1 
\else  \fi

\ifnum \Debug = 1 \usepackage[notref,notcite]{showkeys}
\fi

\ifnum \COLORON = 0 \renewcommand{\color}[1]{}
\fi

\newcommand{\N}{\ensuremath{\mathbb N}}
\newcommand{\R}{\ensuremath{\mathbb R}}
\newcommand{\Z}{\ensuremath{\mathbb Z}}

\newcommand{\cf}{\ensuremath{\mathcal F}}

\newcommand{\Gam}{\ensuremath{\Gamma}}

\newcommand{\Del}{\ensuremath{\Delta}}

\newcommand{\sm}{\backslash}

\newcommand{\sydi}{\triangle}

\newcommand{\seq}[1]{\ensuremath{(#1_n)_{n\in\N}}} 

\newcommand{\g}{\ensuremath{G\ }}
\newcommand{\G}{\ensuremath{G}}

\newcommand{\floor}[1]{\ensuremath{\left\lfloor #1 \right\rfloor}}

\newcommand{\Lr}[1]{Lemma~\ref{#1}}

\newcommand{\Tr}[1]{Theorem~\ref{#1}}

\newcommand{\Prb}[1]{Problem~\ref{#1}}

\newcommand{\Cg}{Cayley graph}

\newcommand{\fe}{for every}

\newcommand{\st}{such that}

\newcommand{\ti}{there is}

\newcommand{\wrt}{with respect to}

\newcommand{\labtequ}[2]{
 \begin{equation} \label{#1} 	\begin{minipage}[c]{0.9\textwidth}  #2 \end{minipage} \ignorespacesafterend \end{equation} }

\newcommand{\mymargin}[1]{
  \marginpar{%
    \begin{minipage}{\marginparwidth}\small%
      \begin{flushleft}%
        {\color{blue}#1}%
      \end{flushleft}%
   \end{minipage}%
  }%
}%

\newcommand{\mySection}[2]{}

\newcommand{\Aut}{Aut}

\newcommand{\pecy}{peripheral cycle}
\newcommand{\ws}{weakly residually finite}

\title{On covers of graphs by Cayley graphs} 
\author{Agelos Georgakopoulos\thanks{Supported by EPSRC grant EP/L002787/1.}
\medskip 
\\
  {Mathematics Institute}\\
 {University of Warwick}\\
  {CV4 7AL, UK}}

\date{}
\begin{document}
\maketitle

\begin{abstract}
We prove that every vertex transitive, planar, 1-ended, graph covers every graph whose balls of radius $r$ are isomorphic to the ball of radius $r$ in $G$ for a sufficiently large $r$. We ask whether this is a general property of finitely presented Cayley graphs, as well as further related questions.
\end{abstract}

\section{Introduction}

We will say that a graph $H$ is \defi{$r$-locally-$G$} if every ball of radius $r$ in $H$
is isomorphic to the ball of radius $r$ in $G$.
The following problem arose from a discussion with Itai Benjamini, and also appears in \cite{BenCoa}.
\begin{problem} \label{pr1}
Does every finitely presented Cayley graph \G\ admit an $r\in \N$ such that \G\ covers
every $r$-locally-$G$ graph?
\end{problem}

The condition of being finitely presented is important here: for example, no such $r$ exists for the standard \Cg\ of the lamplighter group $\Z \wr \Z_2$.

Benjamini \& Ellis \cite{BeElStr} show that $r=2$ suffices for the square grid $\Z^2$, while $r=3$ suffices for the $d$-dimensional lattice (i.e.\ the standard \Cg\ of $\Z^d$ for any $d\geq 3$.



The main result of this paper is 

\begin{theorem}\label{main}
Let \g be a vertex transitive planar 1-ended graph. Then \ti\ $r\in \N$ \st\ \g covers
every $r$-locally-$G$ graph (normally).
\end{theorem} 

Here, we say that a cover $c:V(G)\to V(H)$ is \defi{normal}, if for every $v, w \in V(G)$ such that $c(v) = c(w)$, there is an automorphism $\alpha$ of $G$ such that $\alpha(v) = \alpha(w)$ and $c\circ \alpha = c$. If $c:V(G)\to V(H)$ is a normal cover, then $H$ is a quotient of $G$ by a subgroup of $Aut(G)$, namely the group of `covering transformations'; see \cite[Lemma~16]{BeElStr} for a proof and more details. Normality of the covers was important in \cite{BeElStr}, as it allows one to reduce enumeration problems for graphs covered by lattices to counting certain subgroups of $Aut(G)$.

A natural approach for proving \Tr{main} is by glueing 2-cells to the $r$-locally-$G$ graph $H$ along cycles that map to face-boundaries of $G$ via local isomorphisms to obtain a surface $S_H$, and consider the universal covering map $c: \R^2 \to S_H$. Then $c^{-1}[H]$ is a 1-ended graph $G'$ embedded in $\R^2$ which is also $r$-locally-$G$, and if we could show that $G'$ is isomorphic to $G$ we would be done. The latter statement however turns out to be as hard as \Tr{main}  itself, and in fact we will obtain it as a byproduct of our proof\footnote{I would like to thank Bojan Mohar for suggesting this approach.}.

Let us call an infinite group \defi{\ws}, if all its \Cg s \g have the following property: \fe\ $r\in \N$, \ti\ a finite graph $H$ which is $r$-locally-$G$. It is not hard to prove that every residually finite group is \ws. Indeed, given a \Cg\ \g of a residually finite group $\Gamma$ and some $r$, we can find a homomorphism $h$ from $\Gamma$ to a finite group $\Delta$ which is injective on the ball of radius $r$ around the origin of $G$. Then the \Cg\ of $\Delta$ \wrt\ the generating set $h[S]$, where $S$ is the generating set of $G$, is indeed $r$-locally-$G$. Is the converse statement also true, that is,

\begin{problem}
Is every \ws\ group residually finite?
\end{problem}

If this is true it would yield an alternative definition of residually finite groups. If not, studying the relationship between \ws\ and sofic groups might be interesting.
Similar questions can be asked using graphs covered by \g rather than $r$-locally-$G$ graphs.


\medskip
Benjamini \& Ellis \cite{BeElRan} consider the uniform probability distribution on the  $r$-locally-$G$ graphs with $n$ vertices for $G=\Z^n$, and study properties of this random graphs as $n$ grows. They do so by exploiting normal covers in order to reduce the enumeration of such graphs to the enumeration of certain subgroups of $Aut(G)$, which had previously been studied. \Tr{main} paves the way for the study of the uniformly random $r$-locally-$G$ graph $H_n$ on $n$ vertices, with \g being e.g.\ a regular hyperbolic tessellation. The genus of $H_n$ can easily be seen to be linear in $n$ in our case (while it was always 1 in \cite{BeElRan} for $G=\Z^2$). Glueing metric 2-cells to $H_n$ as described above we obtain a random closed Riemannian surface.  I hope that this topic will be pursued in future work.

\comment{
WRONG:
We remark that it is easy for a cover to be normal:

\begin{theorem}\label{normal}
Let \g be a locally finite graph such that $Aut(G)$ is discrete. Then \ti\ $r\in \N$ \st\ 
every cover $c:V(G)\to V(H)$ of  injectivity radius at least $r$ is normal, where $H$ is an arbitrary graph.
\end{theorem} 
Here, the \defi{injectivity radius} of $c$ is the smallest $r\in \N$ \st\ the restriction of $c$ to the ball of radius $r$ around any vertex is injective.
}

\medskip
Our $r$ in \Tr{main} can be arbitrarily large. It is not clear from our proof whether there is an upper bound depending on the maximum co-degree (i.e.\ length of a face) of $G$ only, or it also depends e.g.\ on the vertex degree. The results of \cite{vapf} might be helpful for answering this question.
\medskip

Tessera and De La Salle (private communication) recently announced a positive answer to \Prb{pr1} under the condition that $Aut(G)$ is discrete, and a counterexample showing that this condition is necessary. 

\section{Preliminaries}

A graph \g is \defi{(vertex) transitive}, if for every two vertices $v,w$ there is an automorphism of \g mapping $v$ to $w$. The group of automorphisms of \g is denoted by $Aut(G)$. We say that $Aut(G)$ is \defi{discrete}, if the stabiliser of each vertex is finite.

A \defi{cover} from a graph $G$ to a graph $H$ is a map $c:V(G)\to V(H)$ such that the restriction of $c$ to the neighbourhood of any vertex of $G$ is a bijection.

\subsection{Planar graphs}

A \defi{plane graph} is a graph \g endowed with a fixed embedding in the plane $\R^2$; more formally, \g is a plane graph if $V(G)\subset \R^2$ and each edge $e\in E(G)$ is an arc between its two vertices that does not meet any other vertices or edges. A graph is \defi{planar} if it admits an embedding in $\R^2$. Note that a given planar graph can be isomorphic (in the graph-therotic sense) to various plane graphs that cannot necessarily be mapped onto each other via a homeomorphism of $\R^2$.

A {\em face} of a planar embedding is a component of the complement of its image, that is, a maximal connected subset of the plane to which no vertex or edge is mapped. The {\em boundary} of a face is the set of edges in its closure. 

\begin{lemma}[\cite{KroInf}] \label{kron}
Let \g be a vertex transitive plane 1-ended graph. Then every face-boundary of \g contains only finitely many edges.
\end{lemma} 

This means that every face-boundary is a cycle of \g in our case.

\comment{
An embedding of $G$ is called {\em consistent} if, intuitively, it embeds every vertex in a similar way in the sense that the group action carries faces to faces. Let us make this more precise.
Given an embedding $\sigma$ of a graph $G$, we consider for every vertex $x$ the embedding of the edges incident with $x$, and define the {\em spin} of $x$ to be the cyclic order of the set $\{xy \mid y\in N(x)\}$ in which $xy_1$ is a successor of $xy_2$ whenever the edge $xy_2$ comes immediately after the edge $xy_1$ as we move clockwise around $x$. 

Call an automorphism $\alpha$ of $G$ {\em spin-preserving} if for every $x\in VG$ the spin of $x\alpha$ is the image of the spin of $x$ in $\sigma$. 
Call it {\em spin-reversing} if for every $x\in VG$ the spin of $x\alpha$ is the reverse of the image of the spin of $x$ in $\sigma$. Call an automorphism {\em consistent} if it is spin-preserving or spin-reversing in $\sigma$. Finally, call the embedding $\sigma$ {\em consistent} if every automorphism of $G$ is consistent in $\sigma$.
}

Given a planar embedding of a graph $G$, we define a \emph{facial path} to be a path of $G$ contained in the boundary of a face. We define a \emph{facial walk} similarly.

The following is a classical result, proved by Whitney \cite[Theorem 11]{whitney_congruent_1932} for finite graphs. It extends to infinite ones by compactness; see \cite{ImWhi}.
\begin{theorem} \label{imrcb}
Let $G$ be a 3-connected graph embedded in the sphere. Then every automorphism of $G$ maps each facial path to a facial path. 
\qed
\end{theorem}

The \emph{connectivity} of a graph is the cardinality of a smallest vertex set whose deletion disconnects the graph. A graph is \emph{$3$-connected} if its connectivity is at least 3.
The next result is due to Babai and Watkins \cite{BaWaCon}, see also~\cite[Lemma~2.4]{baGro}.

\begin{lemma}\label{lem_BW}{\em \cite[Theorem~1]{BaWaCon}}
Let $G$ be a locally finite connected transitive graph  that has precisely one end.
Let $d$ be the degree of any of its vertices.
Then the connectivity of $G$ is at least $3(d+1)/4$.
\qed
\end{lemma}

We deduce from Lemma~\ref{lem_BW} and Theorem~\ref{imrcb} that for every 1-ended transitive planar graph, face-boundaries depend only on the graph and not on any embedding we might choose.

\subsection{Graphs that are locally planar}

Given a graph $H$ that is $r$-locally-$G$, where \g is planar, we would like to be able to talk about `face-boundaries' of $H$, although $H$ is not necessarily planar itself. This can be done by using the notion of a \defi{\pecy}. Recall that a cycle $C= v_0, v_1 \ldots, v_k=v_0$ of a graph $H$ is \defi{induced}, if \g contains no edge from $v_i$ to $v_j$ for $|i-j|>1 (mod\ k)$. A cycle  $C$ is \defi{peripheral} if it is both induced and non-separating. If \g is a connected plane graph then each \pecy\ bounds a face of \G. If \g is also 3-connected, then every face-boundary is peripheral.

A \defi{flag} of a plane graph \g is a triple $\{u,e,F\}$, consisting of a vertex $u$, an edge $e$, and a face-boundary $F$, such that $u\in e \in F$. We denote by $\cf=\cf(G)$ the set of all flags of \G.

Note that by Whitney's theorem, every automorphism of \g can be naturally extended to the flags of \G.

For a vertex $o$ of \G, the \defi{ball} $B_i(o;G)$ of radius $i$ ---also denoted by $B_i(o)$ if \g is fixed--- is the subgraph of \g induced by the vertices at graph-distance at most $i$ from $o$. As we are dealing with planar graphs, it is more convenient to consider the following variant: 
\begin{definition} \label{Di}
We let {\bf $D_k(o;G)$} denote $B_j(o;G)$ for the smallest $j\in \N \cup \{\infty\}$ such  that $B_j(o;G)$ contains every vertex $v\in V(G)$ for which \ti\ a sequence of \pecy s $C_1, \ldots, C_k$ with $o\in V(C_1), v\in V(C_k)$, and $C_i\cap C_{i+1} \neq \emptyset$ \fe\ relevant $i$.
\end{definition}
Note that if \g is planar, then every peripheral cycle bounds a face, and so $j$ is finite. In this case  $D_i(o;G)$ is a ball of \g large enough to contain the ball of radius $i$ of the dual of \G, but the definition also makes sense for non-planar graphs.

\begin{lemma}
Let \g be a vertex transitive plane 1-ended graph, and $o\in V(G)$. Then the face-boundaries of \g containing $o$ coincide with the \pecy s of $D_2(o)$.
\end{lemma}
\begin{proof}
Let $F$ be a face-boundary incident with $o$. By \Lr{kron}, $F$ is a finite cycle $v_1(=o) v_2 \ldots v_k=v_1$. Clearly, $F$ is induced in $D_2(o)$; we will show it is non-separating. It is not hard to prove (see e.g.\ \cite[Lemma~1.1]{BrMoGro}) that $P_i:= D_1(v_i)\sm F$ is a path \fe\  $1\leq i \leq k$. It follows that $F':= \bigcup_{1\leq i \leq k} P_i$ is connected (in fact, it is a cycle). Moreover, $F'$ separates $F$ from $G \sm F$ by construction. To show that $F$ does not separate $D_2(o)$, notice that if $Q$ is a path with both its endvertices outside $F$ with $Q \cap F\neq \emptyset$, then $Q$ meets $F'$ and can be shortcut into a path with the same endvertices avoiding $F$. Thus $F$ is peripheral in  $D_2(o)$.

Conversely, let $F$ be a \pecy\ of $D_2(o)$ containing $o$. Then $F$ is a face-boundary in any embedding of $D_2(o)$ (or \G) as remarked above.
\end{proof}

This lemma justifies the following definition, which allows us to retain our intuition of faces in an $r$-locally-$G$ graph which is not necessarily planar. 
\begin{definition} \label{def1}
Let \g be a vertex transitive plane 1-ended graph, and $H$ a graph which is $r$-locally-$G$ for some $r\geq 2$. We define a \defi{face-boundary} of $H$ to be any \pecy\ of $D_2(v;H)$ incident with $v$ for any $v\in V(H)$. We extend the definition of  a flag, and that of a facial walk, to such graphs $H$ using this notion of face-boundary.
\end{definition}

\subsection{Automorphisms, flags, and fundamental domains}
\Tr{main} is easier to prove when the face-boundaries incident with one (and hence each) vertex have distinct sizes. Complications arise when this is not the case, especially when the automorphism group of \g has non-trivial vertex stabilizers. In order to deal with these complications, we adapt the standard notion of a fundamental domain to our planar setup as follows. We fix a vertex $o\in V(G)$, and define a \defi{fundamental domain} of \g to be a connected sequence of flags of $o$ containing exactly one flag from each orbit of $\Aut(G)$. Here, we say that a sequence $f_1, \ldots, f_k$ of flags of $o$ is \defi{connected}, if $f_i$ is incident with $f_{i+1}$ \fe\ $1\leq i<k$, and we say that $\{o,e,F\}$ is \defi{incident} to $\{o,e',F'\}$ if either $e=e'$ or $F=F'$. For $i\in\N$, we define an \defi{$i$-fundamental domain} of \g similarly except that we replace $\Aut(G)$ by $\Aut(D_i(o))$.

\section{Proof of \Tr{main}}
\begin{lemma}\label{lemnFD}	
There is $n\in\N$ \st\ every $n$-fundamental domain of \g is a fundamental domain.
\end{lemma}
\begin{proof}
The cardinality of an $i$-fundamental domains is monotone increasing with $i$ by the definitions. Since this size is bounded above by twice the degree of \G, a maximum is achieved for some $n$.
\end{proof}

From now on we fix a fundamental domain $\Delta$ of \G. We define a map $\phi: \cf(G) \to \Delta$ by letting $\phi(f)$ be the unique flag in $\Delta$ in the orbit of $f$ under $\Aut(G)$; the existence and uniqueness of such a flag follow from the transitivity of \g and the definition of a fundamental domain. 
By the \emph{colour} of a flag $f$ we will mean the flag $\phi(f)$ of $\Delta$.

Our next observation is that a similar map can be defined on the flags of any $r$-locally-$G$ graph for $r$ at least as large as the $n$ of \Lr{lemnFD}:

\begin{lemma}\label{phipi}	
Let $n\in\N$ be \st\ every $n$-fundamental domain of \g is a fundamental domain, and let $H$ be an $n$-locally-$G$ graph. Then \fe\ $x\in V(H)$, and every two isomophisms $\pi,\pi': D_n(x;H)\to D_n(o;G)$, the compositions $\phi \pi, \phi \pi'$ coincide.
\end{lemma}
\begin{proof}
Suppose, to the contrary, that $\phi \pi(f) \neq \phi \pi'(f)$ for some $f\in \cf(H)$. Then letting $g:=\pi(f)$, we have $\phi \pi \pi^{-1}(g) \neq \phi(g)$. But as $\pi \pi^{-1}\in \Aut(D_n(o;G))$, this contradicts the fact that $\Delta$ is an $n$-fundamental domain of \G, which holds by \Lr{lemnFD}.
\end{proof}

This allows us to define a map $\phi_H: \cf(H) \to \Delta$ by letting $\phi_H(f)$ be the unique flag in $\Delta$ that equals  $\phi \pi (f)$ for some isomophism $\pi: D_n(x;H)\to D_n(o;G)$. Again, the \emph{colour} of a flag $h$ of $H$ is $\phi_H(h)\in \cf(\Delta)$.

\medskip
We let $r:=n+1$ for the rest of this section. 

\begin{lemma}\label{extend}	
Let $H$ be an $r$-locally-$G$ graph, let $v\in V(G), x\in V(H)$. Let $f$ be a flag of $v$ in \g and $h$ a flag of $x$ in $H$ \st\ $\phi(f)=\phi_H(h)$. Then \ti\ a unique isomorphism $i$ from $D_r(v;G)$ to $D_r(x;H)$ \st\ $i(f)=h$.
\end{lemma}
\begin{proof}
Let $\pi: D_r(o;G)  \to D_r(x;H)$ be an isomophism, which exists by the definition of $r$-locally-$G$. Then $\phi\pi(h) = \phi(f)$ by \Lr{phipi}. 

By the definition of \Del\ and \Lr{lemnFD}, \ti\ an automorphism $a$ of \g mapping $v$ to $o$ with $a(f)\in \Del$. Let $a': D_r(v)\to D_r(o)$ be the restriction of $a$ to $D_r(v)$. Then the composition $\pi a'$ is the desired isomorphism  from $D_r(v;G)$ to $D_r(x;H)$.
\end{proof}

\begin{lemma}\label{colours}	
Let $H$ be an $r$-locally-$G$ graph, and let $c$ be an isomorphism from a face-boundary $F$ of $G$ to a a face-boundary of $H$ (recall Definition~\ref{def1}). Suppose that for some flag $f$ of $F$, we have $\phi(f)=\phi_H(c(f))$. Then  \fe\ flag $f'$ of $F$, we have $\phi(f')=\phi_H(c(f'))$.
\end{lemma}
\begin{proof}
By \Lr{extend}, \ti\ an isomorphism $i$ from $D_r(v)$ to $D_r(c(v))$ with $i(f)=c(f)$. In particular, $i(F)=c(F)$ and $i$ extends $c$. Given any $x\in V(F)$, let $i'$ denote the restriction of $i$ to $D_n(x)$, recalling that $n=r-1$ and $n$ satisfies the condition of \Lr{phipi}. As $x$ and $v$ lie on a common face $F$, we have $D_n(x)\subseteq D_r(v)$ and so $i'$ is an isomorphism from $D_n(x)$ to $D_n(i(x))$. By \Lr{phipi} the colour of any flag $g=\{x,e,F\}$ coincides with the colour of $i'(g)$. As $i'(g)=i(g)=c(g)$ (recall $i$ extends $c$), our claim follows.
 
\end{proof}

We can now prove our main result, which strengthens \Tr{main}.

\begin{lemma}\label{Lmain}	
Let $H$ be an $r$-locally-$G$ graph. Let $f=\{v,e,F\}, h=\{x,e',F'\}$ be flags of $G,H$ respectively, \st\ $\phi(f)=\phi_H(h)$. Then \ti\ a unique cover $c$ from $G$ to $H$ \st\ $c(f)=h$. This cover is normal.
\end{lemma}
\begin{proof}
We are going to construct the cover $c$ inductively, starting with the face $F$ of $f$ and then mapping the surrounding faces one by one.

The first step is straightforward: we set $c_0(v)= x$, and let $c_0$ map the remaining vertices of $F$ to $F'$ in the right order, so that $c_0(f)= h$. We remark that $c_0$ preserves colours of flags by \Lr{colours} since it does so for $f$.

For the inductive step, we let $C_0$ be the cycle bounding $F$, and for $i=1,2,\ldots$ we assume that $C_{i-1}$ is a cycle in \g and that we have already defined a map $c_{i-1}$ from  the intersection of \g with the inside of $C$ to $H$ in such a way that the following conditions are all satisfied:
\begin{enumerate}
\item \label{ci} $c_{i-1}$ preserves colours;
\item \label{cii} the restriction of $c_{i-1}$ to $E(v)$ is injective \fe\ $v\in V(G)$, and if $e,e'\in E(v)$ lie in a common face-boundary, then so do $c_{i-1}(e),c_{i-1}(e')$ (in other words, $c_{i-1}$ preserves the cyclic ordering of the edges around any vertex); and
\item \label{ciii} \fe\ edge $e$ in the domain of $c_{i-1}$ (by which we mean that both endvertices of $e$ are in the domain), some face-boundary of $G$ containing $e$ is mapped by  $c_{i-1}$ injectively to a  face-boundary of $H$.
\end{enumerate}

We are going to obtain the cycle $C_{i}$ from $C_{i-1}$ by attaching an incident face-boundary $F_i$. To make sure that every face is mapped at some point, we can fix an enumeration \seq{D} of the face-boundaries of \G. Then, at step $i$ we consider the minimum $n$ such that $D_n$ shares one or more edges with $C_{i-1}$ but does not lie inside $C_{i-1}$, and moreover, $D_n \cap C_{i-1}$ is a path, and let $F_i$ be this $D_n$. 
To see that $F_i$ is well-defined, note that if some $D_j$ satisfies all above requirements except the last one then,  $D_j \cup C_{i-1}$ bounds a region $A$ of $\R^2$ containing only finitely many faces; this is true because every 1-ended planar graph admits an embedding in the plane without accummulation points of vertices \cite{ThomassenRichter}. Each one of these faces $D$ is a candidate for $F_i$, and for those that also fail the requirement that $D \cap C_{i-1}$ is a path, there is a corresponding region $A_D$ strictly contained in $A$. As there are only finitely many such candidates, it is easy to see that at least one of them satisfies all above requirements, and we can choose it as $F_i$. This argument also easily implies that each $D_j$ will be chosen as $F_i$ at some step $i$.

Since $F_i \cap C_{i-1}$ is a path, and it contains an edge, $F_i \sydi C_{i-1}$ is a cycle, which we declare to be $C_{i}$. It remains to extend $c_{i-1}$ to $c_i$ by mapping $F_i \sm C_{i-1}$ to $H$ in a way that preserves flag colours.

Let $w$ be an end-vertex of the path $P:= F_i \cap C_{i-1}$
We claim that there is a unique face-boundary $B$ of $H$ incident with $c_{i-1}(w)$ \st\ (I) $c_{i-1}(P)\subseteq B$, (II) \ti\ an edge of $E(c_{i-1}(w)) \cap B$ not in $c_{i-1}(E(w))$, and (III) $|B|=|F_i|$. To prove this, we will make use of the following observation
\labtequ{obs}{$c_{i-1}$ maps every facial walk $W$ of length 3 in its domain to a facial walk.}
Indeed, let $d, m, g$ be the three edges appearing in $W$ in that order, and let $u,v$ be the endvertices of $m$ incident with $d,g$ respectively. Let $d',g'$ be the other two edges that lie in a common face-boundary with $m$ and are incident with $u,v$ respectively. Note that $m$ is the middle edge of exactly two facial walks of length 3, namely $W$ and $d' m g'$. 

Recall that $c_{i-1}$ preserves adjacency of edges by \ref{cii}, hence the images of $d,d',g,g'$ participate in the 2 facial walks of length 3 in $H$ having $c_{i-1}(m)$ as the middle edge. By \ref{ciii} we know that at least one of the walks $c_{i-1}(d) c_{i-1}(m) c_{i-1}(g)$ and $c_{i-1}(d') c_{i-1}(m) c_{i-1}(g')$ is mapped to a facial walk, and hence, if $d',g'$ are also in the domain of $c_{i-1}$, so is the other by the last remark. This proves \eqref{obs}.
\medskip

From \eqref{obs} we can deduce that $c_{i-1}$ maps every facial walk $W= e_1 e_2 \ldots e_k$, no matter how long, to a facial walk. Indeed, any pair of consequtive edges $e_i e_{i+1}$ in $W$ uniquely determines a face $K_i$ of $H$ containing $c_{i-1}(e_i) c_{i-1}(e_{i+1})$ by \ref{cii} and the fact that $e_i e_{i+1}$ is facial in $G$. But by \eqref{obs}, $K_i=K_{i+1}$ \fe\ relevant $i$ because $c_{i-1}(e_i) c_{i-1}(e_{i+1}) c_{i-1}(e_{i+2})$ is facial.

Applying this to our path $P$, we deduce that $c_{i-1}(P)$ is facial, and we choose $B$ to be the face-boundary it belongs to and, if there is a choice (which only occurs when $P$ is a single edge), contains an edge not in $c_{i-1}(w)$. This automatically satisfies (I) and (II). 

To see that (III) is also satisfied, consider the flag $g:=\{w,e,F_i\}$, where $e$ is the edge of $w$ contained in $F_i \cap C_{i-1}$. This flag is incident with the other flag $g':=\{w,e,D\}$ containing $w$ and $e$, where $D$ lies inside $C_{i-1}$ and is therefore in the domain of $c_{i-1}$. Let $j$ denote the flag $c_{i-1}(g')$, and note that $j$ is incident with $c_i(g)$ along the edge $c_{i-1}(e)$. Now as $c_{i-1}$ preserves colours by \ref{ci}, and the colour of any flag is uniquely determined by the colour of any of its incident flags by our definition of colour, this implies that the colour of $g$ coincides with the colour of the flag $\{c_{i-1}(w), c_{i-1}(e), B\}$ of $H$. In particular, we have $|B|=|F_i|$ by our definition of colour. This completes the proof of our claim.

\medskip
We now obtain $c_i$ by extending $c_{i-1}$ in such a way that $c_i(F_i)=B$. Note that there is a unique such extension as $c_{i-1}$ already maps a non-trivial subpath of $F_i$ to $B$.

By our last remark, $c_i$ preserves the colour of the flag $g$. By \Lr{colours}, $c_i$ preserves the colours of all flags of $F_i$, and as it extends $c_{i-1}$, our inductive hypothesis \ref{ci} that all flag colours are preserved is satisfied.  Condition (II) in the choice of $B$ ensures that \ref{cii} is also satisfied. Finally, \ref{ciii} is satisfied by the construction of $c_i$. 

Thus our inductive hypothesis is preserved.
Letting $c := \bigcup_i c_i$ we obtain a map from $V(G)$ to $V(H)$, which is a cover by \ref{cii}.

\medskip
To see that $c$ is unique, note that $c_0$ was uniquely determined by $f,h$, and at each step $i$, the map $c_i$ was the unique way to extend $c_{i-1}$ while keeping it a candidate for being the restriction of a cover because $B$ was uniquely determined by $F_i$ and $c_{i-1}$.

This uniqueness combined with the definition of $\Delta$ easily implies that $c$ is normal. Indeed, Suppose $c(v)=c(w)=x$ for $v,w\in V(G)$. Let $h$ be a flag of $x$, and let $f_v,f_w$ be the flags of $v,w$ respectively such that $c(f_v)=h= c(f_w)$. Then $\phi(f_v)=\phi_H(h)=\phi(f_w)$. Therefore, there is an automorphism $\alpha$ of $G$ such that $\alpha(v)=w$ and $\alpha(f_v)=f_w$. Note that $c \circ \alpha$ is a cover of $H$ by $G$, and that $(c \circ \alpha)(f_v)=c(f_w)=h$. But as $c(f_v)=h$ too, the uniqueness of $c$ proved above implies $c \circ \alpha = c$ as desired.
\end{proof}


Note that if $G$ is a planar 1-ended vertex transitive graph, and $f,h$ are flags of \g with $\phi(f)=\phi(h)$, then \ti\ an automorphism $a$ of $G$ with $a(f)=h$ by the definition of $\phi$. \Lr{Lmain}, applied with $H=G$, implies that this automorphism is unique.

\comment{ RUBISH:
\end{corollary}
\begin{proof}
Applying \Lr{Lmain} with $H=G$, we obtain a cover $a: V(G) \to V(G)$ with $a(f)=h$. We can easily extend $a$ to a cover $a'$ from $\R^2$ to $\R^2$ by mapping the faces incident to any vertex $v$ to corresponding faces of $a(v)$. Since any such cover $a'$ is a homeomorphism, $a$ is injective, therefore an automorphism.
\end{proof}
}

\comment{
\section{The automorphism group of a planar \Cg}

\begin{theorem}\label{}
Let \g be an 1-ended, planar \Cg\ of a group \Gam. Then $\Aut(G) = \Gam \times D$ where $D$ is a finite cyclic or dihedral group.
\end{theorem} 
\begin{proof} 	

\end{proof}

\begin{lemma}\label{}	
Let \g be a vertex-transitive 1-ended planar graph, and $o\in V(G)$. Then $\Aut(D_r(o))$ is cyclic or dihedral. \mymargin{which r?}
\end{lemma}
\begin{proof}

\end{proof}
}

\section{Further remarks}

One could try to strengthen \Prb{pr1} by demanding that \ti\ a cover arising by taking a group-theoretic quotient of \G, i.e.\ by imposing some further relations to \G, so that the covered graph is a \Cg\ of a quotient of the group of \G. However, the following example shows that this is not possible even in the abelian case: we construct a \Cg\ \g and a family $K=K(l,k)$ of $l-1$-locally-$G$ graphs \st\ \g covers $K$ but $K$ is not even vertex-transitive.

\medskip
{\bf Example:} Let $G'$ be the \Cg\ of $Z \times Z/k$ with the standard generators $(1,0), (0,1)$. Let \g be the graph obtained from the union of two disjoint copies of $G'$ after joining every vertex $x$ of the first copy to every vertex in the second copy that is at the same `height', i.e.\ has the same first coordinate (and so $x$ obtains $k$ new neighbours in the other copy). Easily, \g is a \Cg.

Let $H$ be a toroidal grid of `length' $l$ and `width' $k$ (you may fix $k$ to 4, say). We index the vertices of $H$ as $x_i^j, 0\leq i < l, 0\leq j < k$, so that the neighbours of any $x_i^j$ are $x_{i-1}^j,x_{i+1}^j, x_i^{j-1},x_i^{j+1}$, where all lower indices are $\mod l$ and upper ones $\mod k$. Let $H'$ be a copy of $H$ with its vertices indexed by $y_i^j$ as above. Modify $H'$ into a new graph $H''$ by rerouting one level of its edges: for every $0\leq j < k$, we remove the edge from $x_0^j$ to $x_1^j$ and add an edge from $x_0^j$ to $x_1^{j+1}$. Note that both $H$ and $H''$ look locally like a toroidal grid when $k>>l$, but $H \cup H''$ is not vertex transitive: it has no automorphism mapping $H$ to $H''$. 

Let us now add some edges to $H \cup H''$ to make it connected: for every $i$, we join each of the $k$ vertices in $\{x_i^j \mid 0\leq j < k\}$ to each of the $k$ vertices in $\{x_i^j \mid 0\leq j < k\}$ by a new edge, and let $K$ be the resulting graph.

Note that for every $v,w\in V(K)$, the balls of radius $diam(K) - 1 - \floor{\frac{l}{2}}$ around $v$ and $w$ are isomorphic, yet $K$ is not vertex transitive. 
Still, it is easy to see that \g covers every such $K$.

\comment{
\section{Graphs with discrete automorphism group, and normal covers}

The following is probably well-known
\begin{lemma}
Let \g be a graph \st\ $Aut(G)$ is discrete. Thene there is $r\in \N$ such that every automorphism of $B_o(r)$ extends to an automorphism of $G$.
\end{lemma}
\begin{proof}
If not, then there is an infinite sequence \seq{b} \st\ $b_r$ is an automorphism of $B_o(r)$ that does not extend to an automorphism of $B_o(r+1)$.

\end{proof}
}

\comment{\section{Problems}

A similar problem:
\begin{problem}
Let $G_1,G_2$ be such that any 2-ball of $G_1$ is isomorphic to any 2-ball of $G_2$. Is there a \g that covers both $G_i$ and has isomorphic 2-balls to the $G_i$?
\end{problem}

If I remember well, Itai was also asking a question similar to the above, but \g now is {\em covered} by both Gi, with some additional condition (finite fibers?)
}

\acknowledgement{Most of this work, in particular \Prb{pr1} and \Tr{main}, was triggered by discussions with Itai Benjamini.
I am grateful to David Ellis for the idea of making our covers normal in \Tr{main}.}

\bibliographystyle{plain}
\bibliography{collective}
\end{document}